\documentclass[12pt]{article}
\usepackage{a4wide}
\usepackage{amsmath,amssymb,amsthm}
\usepackage[mathscr]{eucal}
\usepackage[T2A]{fontenc}
\usepackage{mathtools}

\author{Ievgen~V.~Bondarenko, Rostyslav~V.~Kravchenko\footnote{This author was supported by the NSF grant number 0503688 and the ERC starting grant GA 257110 ТRaWGУ.}}
\title{\textbf{Finite-state self-similar actions of nilpotent groups}}

\newcommand{\nucl}{\mathcal{N}}
\newcommand{\srL}{\mathscr{L}}
\newcommand{\mbQ}{\mathbb{Q}}
\newcommand{\phiL}{\widehat{\phi}}

\newtheorem{theorem}{Theorem}
\newtheorem{proposition}[theorem]{Proposition}

\newtheorem{lemma}{Lemma}
\theoremstyle{definition}

\begin{document}

\maketitle

\begin{abstract}
Let $G$ be a finitely generated torsion-free nilpotent group and $\phi:H\rightarrow G$ be a surjective
homomorphism from a subgroup $H<G$ of finite index with trivial $\phi$-core. For every choice of coset
representatives of $H$ in $G$ there is a faithful self-similar action of the group $G$ associated with
$(G,\phi)$. We are interested in what cases all these actions are finite-state and in what cases there
exists a finite-state self-similar action for $(G,\phi)$. These two
properties are characterized in terms of the Jordan normal form of the corresponding automorphism $\phiL$ of the Lie algebra of the Mal'cev completion of $G$.\\

\noindent \textbf{Mathematics Subject Classification 2000}: 20F65, 20F18

\vspace{0.1cm}\noindent \textbf{Keywords}: self-similar action, nilpotent group,
finite-state action, automaton group

\end{abstract}

\section{Introduction}

An action of a group $G$ on the set $X^{*}$ of words over a finite alphabet $X$ is called
\textit{self-similar} if for every $x\in X$ and $g\in G$ there exist $y\in X$ and $h\in G$ such that
$g(xv)=yh(v)$ for all words $v\in X^{*}$. Every self-similar action of a group on $X^{*}$ implies the
action of the group by automorphisms on the $|X|$-regular rooted tree with vertex set $X^{*}$ and edges
$(v,vx)$ for all $v\in X^{*}$, $x\in X$. In terms of automorphisms of rooted trees, self-similarity of
the action means that the action of every element of the group ``inside'' every $|X|$-regular rooted
subtree can be represented by an element of the group. Self-similar group actions appear naturally in
many areas of mathematics and have applications to holomorphic dynamics, fractal geometry,
combinatorics, automata theory, etc. (see \cite{self_sim_groups} and the references therein).

The theory of self-similar group actions can be regarded as the study of positional numeration systems
on groups. Bases in these number systems are virtual endomorphisms of groups. A \textit{virtual
endomorphism} of a group $G$ is a homomorphism $\phi:H\rightarrow G$ from a subgroup $H<G$ of finite
index to $G$. To produce a self-similar action of the group $G$ with ``base'' $\phi$ we need to choose
a set $D$ of coset representatives of $H$ in $G$ called a \textit{digit set}. Let us enumerate the
elements of $D$ by the letters of an alphabet $X$ (here $|X|=[G:H]$ and $D=\{q_x, x\in X\}$). The
\textit{self-similar action} of $G$ on the space $X^{*}$ \textit{associated to the triple} $(G,\phi,D)$
is constructed as follows. Every element of the group stabilizes the empty word. For every $x\in X$ and
$v\in X^{*}$ the action of an element $g\in G$ is defined recursively by the rule
\begin{equation}\label{eq_action_given_by_virt_end}
g(xv)=yh(v) \ \mbox{ with } \ h=\phi(q_y^{-1}gq_x),
\end{equation}
where $y\in X$ is the unique letter such that $q_y^{-1}gq_x\in H$. The constructed action may be not
faithful. The kernel of the action does not depend on the choice of the set $D$ and is equal to the
maximal normal $\phi$-invariant subgroup of $G$ called the \textit{$\phi$-core}
(\cite[Proposition~2.7.5]{self_sim_groups}).

%It is usually assumed that $\phi$ is surjective; in this case the
%associated self-similar actions are called
%\textit{self-replicating} (recurrent).

%so that every element $g\in G$ has ``expansion'' with base $\phi$

Conversely, one can associate a virtual endomorphism  with every faithful self-similar action
$(G,X^{*})$ as follows. The element $h$ from the definition of self-similar action is called the
\textit{state} of $g$ at $x$ and is denoted by $g|_x$ (this element is unique if the action is
faithful); iteratively we define the state of $g$ at every word by the rule $g|_{x_1x_2\ldots
x_n}=g|_{x_1}|_{x_2}\ldots |_{x_n}$. For every letter $x\in X$ the stabilizer $St_G(x)$ has finite
index in $G$ and then the map $\phi_x:St_G(x)\rightarrow G$ given by $\phi_x(g)=g|_x$ is a virtual
endomorphism of $G$. If in addition the action $(G,X^{*})$ is \textit{self-replicating} (recurrent),
i.e., $\phi_x$ is surjective and $G$ acts transitively on $X$, then $(G,X^{*})$ can be given by the
triple $(G,\phi_x,D)$ for some choice of the digit set~$D$.

As a simple example, consider the group $\mathbb{Z}$ with the homomorphism
$\phi:2\mathbb{Z}\rightarrow\mathbb{Z}$, $\phi(2a)=a$, and choose the digit set $D=\{0,1\}$, which is
also used as an alphabet with a slight abuse in notations. The associated self-similar action
corresponds to the binary number system on $\mathbb{Z}$. We have $a|_{x_1x_2\ldots x_n}=b$ and
$a(x_1x_2\ldots x_n)=y_1y_2\ldots y_n$ for $a,b\in\mathbb{Z}$ if and only if
\[
a=(y_1-x_1)+2(y_2-x_2)+2^2(y_3-x_3)+\ldots+2^{n-1}(y_n-x_n)+2^nb.
\]
In particular, if $a|_{00\ldots 0}=0$ then the image $a(00\ldots 0)$ is the usual binary expansion of
$a$.

%Especially important self-similar actions are finite-state actions in which the elements of the group can be given by
%finite automata.

%Another language dealing with self-similar action of groups is
%through groups generated by automata.

Self-similar group actions are closely related to groups generated by automata. Groups generated by
finite automata correspond to finite-state self-similar actions of finitely generated groups. Recall
that a faithful self-similar action $(G,X^{*})$ is called \textit{finite-state} if for every $g\in G$
the set of its states $\{g|_v : v\in X^{*}\}$ is finite. Then a finitely generated group has a faithful
finite-state self-similar action if and only if it can be generated by a finite automaton. One of the
fundamental questions in this theory is what groups possess finite-state self-similar actions, i.e.,
can be realized by finite automata. This property was proved for free abelian groups $\mathbb{Z}^n$
\cite{1/2end}, Grigorchuk group \cite{GNS:ADG}, $GL_n(\mathbb{Z})$ \cite{GlnZ}, lamplighter group
\cite{GZ:lampl}, free groups and free products of cyclic groups of order $2$ \cite{free_gr},
Baumslag-Solitar groups $B(1,m)$ \cite{bs:solvable}, certain nilpotent groups \cite{virt_nilp}, etc.

The finite-state self-similar actions of $\mathbb{Z}^n$ can be characterized in term of the associated
virtual endomorphism as shown by Nekrashevych and Sidki in \cite{1/2end} (see also
\cite[Theorem~2.12.1]{self_sim_groups}). A virtual endomorphism of $\mathbb{Z}^n$ is uniquely extended
to a linear operator of $\mathbb{R}^n$. Then a faithful self-replicating self-similar action of
$\mathbb{Z}^n$ with virtual endomorphism $\phi$ is finite-state if and only if the spectral radius of
$\phi$ is less than $1$. In particular, there is no dependence on the choice of coset representatives.

In this paper we consider self-similar actions of finitely generated torsion-free nilpotent groups. The
main goal is to generalize the above mentioned result of Nekrashevych and Sidki. However, self-similar
actions of nilpotent groups have a new level of complexity comparing to the actions of abelian groups.
For example, a nilpotent group with fixed virtual endomorphism may have a faithful finite-state
self-similar action for one choice of coset representatives and be not finite-state for another choice
(see example with Heisenberg group in Section~\ref{Section_Example}). Hence we need to answer two
questions: Under what conditions on $\phi$ does there exist a finite-state action for $(G,\phi)$? Under
what conditions does every action for $(G,\phi)$ finite-state?

Let $G$ be a finitely generated torsion-free nilpotent group with surjective virtual endomorphism
$\phi:H\rightarrow G$. Since we are interested in faithful self-similar actions of the group $G$, we
assume that $\phi$-core is trivial. One can check the triviality of $\phi$-core by the following
statement.
\begin{proposition}[Corollary~1 in \cite{virt_nilp}]\label{prop_Sidki}
Let $G,H,\phi$ be as above. If $\phi$-core is trivial, then $\phi$ is injective. If $\phi$ is
injective, then $\phi$-core is trivial if and only if the virtual endomorphism $\phi|_{Z(H)}:Z(H)\to
Z(G)$ of the center $Z(G)$ of the group $G$ has trivial core.
\end{proposition}
The proposition implies that $\phi$ is also injective and hence it is an isomorphism. The same
proposition says that if we know that $\phi$ is an isomorphism then $\phi$-core is trivial if and only
if $\phi|_{Z(H)}$-core is trivial. Since $Z(G)$ is abelian, by
\cite[Proposition~2.9.2]{self_sim_groups} the $\phi|_{Z(H)}$-core is trivial if and only if no
eigenvalue of $\phi|_{Z(H)}$ is an algebraic integer, which can be effectively checked.

By a theorem of Mal'cev (see \cite{malcev} and the next section) there exists the unique real nilpotent
connected and simply connected Lie group $L$, {\it Mal'cev completion} of $G$, such that the group $G$
is a discrete subgroup of $L$ and the topological space $L/G$ is compact. Since $H$ is a subgroup of
finite index, the isomorphism $\phi:H\rightarrow G$ lifts to an automorphism of the Lie group $L$ also
denoted by $\phi$. Let $\mathscr{L}$ be the Lie algebra of $L$ and denote by $\phiL$ the differential
of $\phi$ at the identity, which is an automorphism of $\mathscr{L}$. Then the existence of
finite-state self-similar action of the group $G$ can be characterized in terms of the Jordan normal
form of $\phiL$.

\begin{theorem}\label{thm_main1}
Let $G$ be a finitely generated torsion-free nilpotent group. Let $(G,X^{*})$ be a faithful
self-replicating self-similar action with virtual endomorphism $\phi$ (associated to some letter $x\in
X$). If the action $(G,X^{*})$ is finite-state then the spectral radius of $\phiL$ is not greater than
$1$ and for every eigenvalue of modulus $1$ the associated Jordan cells in the Jordan normal form of
$\phiL$ have size~$1$. Conversely, if the automorphism $\phiL$ satisfies the previous condition then
there exists a finite-state self-similar action of $G$ with virtual endomorphism~$\phi$.
\end{theorem}

One can restate the theorem as follows. Let $\phi$ be a surjective virtual endomorphism of $G$ with
trivial core. There exists a digit set $D$ with $e\in D$ such that the self-similar action associated
to $(G,\phi,D)$ is finite-state if and only if the Jordan normal form of $\phiL$ satisfies the
condition in the theorem.

\begin{theorem}\label{thm_main2}
Let $G$ be a finitely generated torsion-free nilpotent group, and let $\phi$ be a surjective virtual
endomorphism of $G$ with trivial core. Every self-similar action of $(G,\phi)$ is finite-state if and
only if the spectral radius of $\phiL$ is less than $1$.
\end{theorem}

In particular, if the Jordan normal form of $\phiL$ satisfies the condition in Theorem~\ref{thm_main1}
and $\phiL$ has an eigenvalue of modulus $1$, then the pair $(G,\phi)$ possesses both finite-state and
non-finite-state self-similar actions. This situation cannot happen for abelian group $\mathbb{Z}^n$,
because if a virtual endomorphism of $\mathbb{Z}^n$ viewed as a linear map has an eigenvalue of modulus
$1$ then it has a non-trivial core.

We should also note that while in this paper we are dealing with torsion free nilpotent groups by
viewing them as lattices in their Mal'cev completions, Michael Kapovich in \cite{kap} has considered
the case of a lattice in linear semisimple algebraic Lie group and proved that it admits a faithful
self-similar action if and only if it is virtually isomorphic to an arithmetic lattice. It is an
interesting open question when such a self-similar action is finite-state (see
\cite[Question~17]{kap}).

\section{Preliminary information about nilpotent Lie groups}

In this section we remind some facts from the theory of nilpotent Lie groups which we will need for the
proof. Let $L$ be a connected and simply connected real nilpotent Lie group and $\mathscr{L}$ be its
Lie algebra. The exponential map $\exp: \mathscr{L}\to L$ is a diffeomorphism and one can define $\log:
L\to \srL$ as the inverse of~$\exp$.

\vspace{0.2cm}\textbf{Uniform subgroups.} We say that $\srL$ has a \textit{rational structure} if there
is a rational Lie subalgebra $\srL_\mbQ$ such that $\srL=\srL_\mbQ\otimes\mathbb{R}$. A subgroup $G$ of
$L$ is called {\it uniform} if it is discrete and $L/G$ is compact. The next theorem connects these two
notions.
\begin{theorem}\label{bookthm1}(Theorem~5.1.8 (a) in \cite{repr_nilpotent_Lie})
Let $L$ be a nilpotent Lie group with Lie algebra $\mathscr{L}$. If $L$ has a uniform subgroup $G$,
then $\mathscr{L}$ has a rational structure such that $\srL_{\mbQ}=\mathbb{Q}\text{-span}(\log G)$.
\end{theorem}
This theorem allows us, given a uniform subgroup $G$ of $L$, to fix the rational structure on $\srL$ by
defining $\srL_{\mbQ}=\mathbb{Q}\text{-span}(\log G)$. The next theorem tells when two different
uniforms subgroups define the same rational structure.
\begin{theorem}\label{bookthm2}(Theorem~5.1.12 in \cite{repr_nilpotent_Lie})
Let $G_1$, $G_2$ be uniform subgroups in a nilpotent Lie group $L$. Then $G_1$, $G_2$ determine the
same rational structure in $\srL$ if and only if they are commensurable: $G_1\cap G_2$ of finite index
in both $G_1$ and $G_2$.
\end{theorem}
We will often consider subgroups of $G$ of finite index in $G$. It is clear they are still uniform
subgroups of $L$, hence the above theorem tells that such a subgroup defines the same rational
structure as $G$.

A uniform subgroup $G$ of $L$ is called a {\it lattice subgroup} if $\log G$ is an additive subgroup of
$\srL$.
\begin{theorem}\label{bookthm3}(Theorem~5.4.2 (a) in \cite{repr_nilpotent_Lie})
Let $G$ be a uniform subgroup of a nilpotent Lie group $L$, then $G$ contains a lattice subgroup $G_0$
of finite index.
\end{theorem}

\vspace{0.2cm}\textbf{Mal'cev completion.} Given a finitely generated torsion-free nilpotent group $G$,
Mal'cev completion embeds $G$ as a uniform subgroup into connected and simply connected real nilpotent
Lie group. This completion can be briefly described as follows. Since $G$ is nilpotent and torsion
free, it is possible to add subgroups to its upper central series so that the obtained series, say
$\{e\}=G_1\leq\dots \leq G_{m+1}=G$, has infinite cyclic factors. By fixing a choice of $a_i\in G$ such
that $a_i G_i$ generates $G_{i+1}/G_i$ for $i=1,\dots, m$ we can write any $g\in G$ in the form
$g=a_{m}^{t_{m}}\dots a_1^{t_1}$ for some uniquely defined $t_i\in\mathbb{Z}$, called {\it Mal'cev
coordinates} of $g$. Let $f:\mathbb{Z}^m\times\mathbb{Z}^m\to\mathbb{Z}^m$ and
$g:\mathbb{Z}^m\to\mathbb{Z}^m$ be the maps that correspond to the product and taking inverse in $G$ in
terms of coordinates. It can be shown that $f,g$ are polynomial, and hence have a natural extension
from $\mathbb{Z}^m$ to $\mathbb{R}^m$. Moreover, they still satisfy the same laws, and hence define a
nilpotent Lie group structure on $\mathbb{R}^m$ such that the subset $\mathbb{Z}^m\subset\mathbb{R}^m$
with induced group operations is a subgroup isomorphic to $G$.

\vspace{0.2cm}\textbf{Left invariant metrics on Lie groups}. We recall the following construction of
the left invariant Riemannian metrics on a Lie group $L$. Let $T_g L$ be the tangent space to $L$ at
the point $g$, and recall that $\srL=T_e L$.  Denote by $l_g:T_e L\to T_g L$ the differential at the
point $e$ of the map $x\mapsto gx$. Then any euclidean norm on $\srL=T_e L$ gives rise to a left
invariant Riemannian metric on $L$ by the rule $\|w\|=\|l_g^{-1}w\|$ for any $g\in L$ and $w\in T_g L$.
Then the length of a differentiable curve $f(t)\in L$, $t\in [0,1]$ is defined as $\int_0^1 \|f'(t)\| d
t$, and the distance $d(g,g')$ between $g,g'\in L$ is the infimum of lengths of all curves that connect
$g$ to $g'$. Note that the length of the curve $f(t)$ is the same as the length of the shifted curve
$g_0f(t)$ for any $g_0\in L$, and hence $d(g_0g,g_0g')=d(g,g')$.

\section{Proof of Theorems~\ref{thm_main1} and \ref{thm_main2}}
Recall that we have a finitely generated torsion-free nilpotent group $G$ and an isomorphism $\phi:H\to
G$ from a subgroup $H<G$ of finite index with trivial $\phi$-core. The Mal'cev completion $L$ of $G$ is
a nilpotent Lie group, connected and simply connected, and $G$ is a uniform subgroup of it. Thus all
the theorems from the previous section apply. Note also that since $H$ is a subgroup of finite index in
$G$, it is also a uniform subgroup of $L$. Therefore $L$ is also Mal'cev completion of $H$, by the
uniqueness property of Mal'cev completion. The isomorphism $\phi:H\to G$ lifts to a diffeomorphic
automorphism of $L$ and its differential $\phiL$ at $e$ is an automorphism of Lie algebra $\srL$. Under
these notations we have
\begin{equation}
\label{commi} \phiL(\log(g))=\log(\phi(g))\mbox{ for all }g\in{}L.
\end{equation}

Before proving the theorems, let us show first that if $\phiL$ is contracting, i.e., its spectral
radius is less than one, then the associated self-similar actions of $G$ (independently on the choice
of coset representatives) has a special contracting property. Namely, a self-similar action $(G,X^{*})$
is called \textit{contracting} if there exists a finite set $\nucl\subset G$ with the property that for
every $g\in G$ there exists $n\in\mathbb{N}$ such that $g|_{v}\in\nucl$ for all words $v\in X^{*}$ of
length $\geq n$. Contracting property of a self-similar action doesn't depend on the choice of coset
representatives and can be characterized just in terms of the associated virtual endomorphism (see
\cite[Proposition~2.11.11]{self_sim_groups}).

\begin{proposition}\label{prop_contracting}
If $\phiL$ is contracting, then every self-similar action of $(G,\phi)$ is contracting.
\end{proposition}
\begin{proof}
Let $(G,X^{*})$ be a self-similar action associated to the triple $(G,\phi,D)$. Take any element $g\in
G$ and for every word $x_1x_2\ldots x_n\in X^{*}$ consider the state
\begin{eqnarray}\label{eqn_proof_prop_g_x1x2=}
g|_{x_1x_2\ldots x_n}&=&\phi(q^{-1}_{y_n}\ldots \phi(q^{-1}_{y_2} \phi(q^{-1}_{y_1} g q_{x_1}) q_{x_2})
\ldots q_{x_n})\\ &=& \phi(q^{-1}_{y_n}) \ldots \phi^{n-1}(q^{-1}_{y_2}) \phi^n(q^{-1}_{y_1})\
\phi^n(g)\ \phi^{n}(q_{x_1}) \phi^{n-1} (q_{x_2}) \ldots \phi(q_{x_n}),\nonumber
\end{eqnarray}
where $y_1y_2\ldots y_n=g(x_1x_2\ldots x_n)$. We use the following general lemma:

\begin{lemma}
Let $\lambda>0$ be bigger than the spectral radius of $\phiL:\srL\to\srL$. Then the topology on $L$ is
induced by a left invariant metric $d:L\times L\to\mathbb{R}_+$ such that $d(\phi(g),\phi(g'))\leq
\lambda d(g,g')$ for all $g,g'\in L$.
\end{lemma}
\begin{proof}
It is easy to see that there is an euclidian norm on $\srL$ such that $\|\phiL(v)\|\leq\lambda\|v\|$
for any $v\in\srL$. Then there is a unique left invariant Riemannian metric on $L$, induced by this
norm. Moreover, by left invariance, for any vector $w$ tangent to $L$ we have that
$\|\phiL(w)\|\leq\lambda\|w\|$. Integrating, we obtain that $d(\phi(g),\phi(g'))\leq \lambda d(g,g')$
for all $g,g'\in L$.
\end{proof}
By the lemma there is $\lambda<1$ and a left-invariant metric $d$ on $L$ such that
 \begin{equation*}
 \begin{aligned}
& d(e,\phi^{n}(q_{x_1}) \phi^{n-1} (q_{x_2}) \ldots \phi(q_{x_n}))\leq\\ \leq\ &
d(e,\phi^n(q_{x_1}))+d(\phi^n(q_{x_1}),\phi^n(q_{x_1})\phi^{n-1}(q_{x_2}))+\dots\\&
\ldots +d(\phi^{n}(q_{x_1}) \phi^{n-1} (q_{x_2}) \ldots \phi^2(q_{x_{n-1}}),\phi^{n}(q_{x_1}) \phi^{n-1} (q_{x_2}) \ldots \phi(q_{x_n}))\leq\\
\leq\ & \lambda^n d(e,q_{x_1})+\dots+\lambda d(e,q_{x_n})\leq\frac{\lambda}{1-\lambda}\max_{i}
d(e,q_{x_i})=:R.
\end{aligned}
 \end{equation*}
Hence the set of all products of the form $\phi^{n}(q_{x_1}) \phi^{n-1} (q_{x_2}) \ldots \phi(q_{x_n})$
belongs to a ball of finite radius $R$ in $L$, which is a compact set. Define the finite subset
$\nucl\subset G$ as the set of all group elements in the ball of radius $2R+1$ around identity in $L$.
Since $\phiL$ is contracting, $d(e,\phi^n(g))\to 0$ as $n\to\infty$ and the product in
(\ref{eqn_proof_prop_g_x1x2=}) belongs to the ball of radius $2R+1$ for all large enough
$n\in\mathbb{N}$. Hence $g|_{x_1x_2\ldots x_n}\in\nucl$ and the action is contracting.
\end{proof}

It is obvious from definition that every contracting action is finite-state. Therefore
Theorem~\ref{thm_main2} can be restated as follows: under the conditions of the theorem, all
self-similar actions for $(G,\phi)$ are finite-state if and only if at least one of them (and hence
every) is contracting.

%The automorphisms $\phi$ of $L$ and $\mathscr{L}$ satisfy

%Hence $\phi$ is a square matrix with rational entries (in the basis of $\mathscr{L}_{\mathbb{Q}}$), and $[G:H]=\det
%\phi^{-1}$.

%A self-similar action $(G,X^{*})$ is called \textit{contracting} if there exists a finite set $\nucl\subset G$ with the
%property that for every $g\in G$ there exists $n\in\mathbb{N}$ such that $g|_{v}\in\nucl$ for all words $v\in X^{*}$ of
%length $\geq n$. Every contracting action is finite-state by definition. A self-replicating self-similar action is
%contracting if and only if the associated virtual endomorphism has spectral radius less than $1$ (see
%\cite[Proposition~2.11.11]{self_sim_groups}).

\begin{proof}[Proof of sufficiency in Theorem~\ref{thm_main1}]
The assumption on the Jordan normal form of $\phi$ implies the following crucial property: for every
$g\in L$ the sequence $\phi^n(g)$ is bounded (i.e., belongs to a compact set).

It can be shown, using induction on the length of lower central series of $L$ that
$[G:H]=\det\phiL^{-1}$.

The Lie algebra $\mathscr{L}$ decomposes in the direct sum
$\mathscr{L}=\mathscr{L}_c\oplus\mathscr{L}_r$, where $\mathscr{L}_c$ is a $\phi$-invariant subalgebra
such that $\phi|_{\mathscr{L}_c}$ has spectral radius less than $1$ (contracting), and the spectrum of
$\phi|_{\mathscr{L}/ \mathscr{L}_c}$ consists only of numbers of modulus~$1$. Consider the
$\phi$-invariant subgroup $L_c=\exp(\mathscr{L}_c)$ of the Lie group $L$ that corresponds to the
subalgebra $\mathscr{L}_c$. One can define $L_c$ directly as the subgroup of all $g\in L$ such that
$\phi^n(g)\rightarrow e$ as $n\rightarrow\infty$. Define the group $G_c=G\cap L_c$ and its subgroup
$H_c=H\cap L_c=\phi^{-1}(G_c)< G_c$ of finite index $[G_c:H_c]=\det(\phi|_{\mathscr{L}_c})^{-1}$.

%Then
%$\phi|_{H_c}:H_c\rightarrow G_c$ is a contracting isomorphism %and every self-similar action of
%$(G_c,\phi|_{H_c})$ is contracting.

Notice that $\det (\phi|_{\mathscr{L_r}/ \mathscr{L}_c})$ is positive as since $ \det(\phi)=\det
(\phi|_{\mathscr{L_r}/ \mathscr{L}_c}) \det(\phi|_{\mathscr{L}_c})$, and, at the same time, it is a
product of numbers of modulus $1$. Hence $\det (\phi|_{\mathscr{L_r}/ \mathscr{L}_c})=1$ and we get
$[G:H]=\det(\phi)^{-1}=\det(\phi|_{\mathscr{L}_c})^{-1}=[G_c:H_c]$.

Take any coset representatives $q_1,q_2,\ldots,q_d$ for $H_c$ in $G_c$. Since $H\cap G_c=H_c$ and
$[G:H]=[G_c:H_c]$, the elements $q_1,q_2,\ldots,q_d$ are also coset representatives of $H$ in $G$. Let
us consider the associated self-similar action $(G,X^{*})$ given by
Equation~(\ref{eq_action_given_by_virt_end}). Take any element $g\in G$ and for every word
$x_1x_2\ldots x_n\in X^{*}$ consider the state
\begin{eqnarray}\label{eqn_proof_suf_g_x1x2=}
g|_{x_1x_2\ldots x_n}&=&\phi(q^{-1}_{y_n}\ldots \phi(q^{-1}_{y_2} \phi(q^{-1}_{y_1} g q_{x_1})
q_{x_2})  \ldots q_{x_n})\\ &=& \phi(q^{-1}_{y_n}) \ldots \phi^{n-1}(q^{-1}_{y_2})
\phi^n(q^{-1}_{y_1})\ \phi^n(g)\ \phi^{n}(q_{x_1}) \phi^{n-1} (q_{x_2}) \ldots \phi(q_{x_n}),\nonumber
\end{eqnarray}
where $y_1y_2\ldots y_n=g(x_1x_2\ldots x_n)$. The sequence $\phi^n(g)$ is bounded in $L$. The elements
$q_{x_1},q_{x_2},\ldots, q_{x_n}$ are taken from a finite subset of the $\phi$-invariant subgroup
$L_c$. Then the set of all products of the form $\phi^{n}(q_{x_1}) \phi^{n-1} (q_{x_2}) \ldots
\phi(q_{x_n})$ belongs to a compact subset of $L_c$ (see the proof of
Proposition~\ref{prop_contracting}). Since the product in (\ref{eqn_proof_suf_g_x1x2=}) belongs to the
lattice $G$, it assumes a finite number of values. Hence the action $(G,X^{*})$ is finite-state.
\end{proof}

%\red{We use the following general lemma:
%\begin{lemma}\label{l1}
%Let $\lambda>0$ be bigger than the spectral radius of $\phi:\srL\to\srL$. Then the topology on $L$ is induced by a left invariant metric $d:L\times L\to\mathbb{R}_+$ such that $d(\phi(g),\phi(g'))\leq \lambda d(g,g')$ for all $g,g'\in L$.
%\end{lemma}
%\begin{proof}
%It is easy to see that there is an euclidian norm on $\srL$ such that $\|\phi(v)\|\leq\lambda\|v\|$ for any $v\in\srL$. Then there is a unique left invariant Riemannian metric on $L$, induced by this norm. Moreover, by left invariance, for any vector $w$ tangent to $L$ we have that $\|\phi(w)\|\leq\lambda\|w\|$. Integrating, we obtain that $d(\phi(g),\phi(g'))\leq \lambda d(g,g')$ for all $g,g'\in L$.
%\end{proof}
%By the lemma there is $\lambda<1$ and a left-invariant metric $d$ on $L_c$ such that
% \begin{equation*}
% \begin{aligned}
%& d(e,\phi^{n}(q_{x_1}) \phi^{n-1} (q_{x_2}) \ldots \phi(q_{x_n}))\leq
% d(e,\phi^n(q_{x_1}))+d(\phi^n(q_{x_1}),\phi^n(q_{x_1})\phi^{n-1}(q_{x_2}))+\dots\\&
%+d(\phi^{n}(q_{x_1}) \phi^{n-1} (q_{x_2}) \ldots \phi^2(q_{x_{n-1}}),\phi^{n}(q_{x_1}) \phi^{n-1} (q_{x_2}) \ldots \phi(q_{x_n}))\leq\\
%&\leq\lambda^n d(e,q_{x_1})+\dots+\lambda d(e,q_{x_n})\leq\frac{\lambda}{1-\lambda}\max_{i} d(e,q_{x_i}).
%\end{aligned}
% \end{equation*}}

%\textbf{??? Remark.} The subgroup $G_c$ of $G$ constructed in the proof is normal and $\phi$-invariant. We can consider
%the quotient group $G/G_c$ with the induced virtual endomorphism, but its self-similar actions are trivial, i.e., the
%$\phi$-core of $\phi^{-1}(G/G_c)$ is equal to the whole group.

\begin{proof}[Proof of necessity in Theorem~\ref{thm_main1}]
Let $(G,X^{*})$ be a finite-state self-similar action with virtual endomorphism $\phi$ associated to
the letter $x_1\in X$, i.e., $\phi=\phi_{x_1}$ and $H=St_G(x_1)$. Let $\{q_1=e,q_2,\ldots,q_d\}$ be the
corresponding set of coset representatives.

\begin{lemma}
The eigenvalues of $\phiL$ have modulus $\leq 1$. Moreover, every eigenvalue of modulus $1$ is a root
of unity.
\end{lemma}
\begin{proof}
Put $\mathscr{L}^{(0)}=\mathscr{L}$ and let $\mathscr{L}^{(i)}=[\mathscr{L},\mathscr{L}^{(i-1)}]$ be
the lower central series of the Lie algebra $\mathscr{L}$. Since $\phiL$ is an automorphism of
$\mathscr{L}$ it preserves every term $\mathscr{L}^{(i)}$ and induces an automorphism $\varphi_i$ on
the quotient $\mathscr{L}^{(i)} / \mathscr{L}^{(i+1)}$. The spectrum of $\phiL$ is a union of the
spectra of $\varphi_{i}$. At the same time, every linear map $\varphi_i$ is a quotient of the tensor
product $\varphi_{0}\otimes\varphi_{0}\otimes\ldots \otimes\varphi_{0}$ (see
\cite[Theorem~3.1]{book:nilpotent}). Hence it is enough to prove the statement for the
automorphism~$\varphi_{0}$.

Let $\lambda$ be an eigenvalue of $\varphi_0$. Take a basis of $\mathscr{L}/[\mathscr{L},\mathscr{L}]$
in which $\varphi_0$ has Jordan normal form, and consider the coordinate of vectors in this basis that
corresponds to an eigenvector with eigenvalue $\lambda$. There exists a linear map
$\xi:\mathscr{L}\rightarrow \mathbb{C}$ such that $\xi([\mathscr{L},\mathscr{L}])=0$ and
$\xi(\phi(l))=\lambda\xi(l)$ for all $l\in\mathscr{L}$. We compose $\xi$ with the logarithmic map
$\log:L\rightarrow\mathscr{L}$ and denote the composition also by $\xi$. Note that
$\log(g_1g_2)=\log(g_1)+\log(g_2)\bmod{[\srL,\srL]}$. Thus we have a map $\xi:L\rightarrow\mathbb{C}$
such that $\xi(g_1g_2)=\xi(g_1)+\xi(g_2)$ and $\xi(\phi(g))=\lambda\xi(g)$. The rest of the proof is
very similar to the proof of Theorem~2.12.1 in \cite{self_sim_groups}, so we only sketch it here.

Since $G$ is a lattice in $L$, there exists $g\in G$ such that $\xi(g)\neq 0$. Let us consider the
states $g|_v$ for $v\in X^{*}$. By Equation~(\ref{eq_action_given_by_virt_end}) we have
$g|_x=\phi(q_{g(x)}^{-1}g q_x)$ for every $x\in X$. Then
\begin{align*}
&\xi(g|_x)=\lambda\xi(g)+\lambda(\xi(q_x)-\xi(q_{g(x)}))=\lambda\xi(g)+d_x,
\end{align*}
where $d_x=\lambda(\xi(q_x)-\xi(q_{g(x)}))$.

Suppose $|\lambda|>1$. Since $\sum_{x\in X}d_x=0$, it follows that there exits $x_1\in X$ such that
$|\xi(g|_{x_1})|>|\xi(g)|$. Thus we can iteratively construct letters $x_n\in X$ such that
$|\xi(g|_{x_1x_2\ldots x_{n+1}})|>|\xi(g|_{x_1x_2\ldots x_n})|$ for each $n$. Hence $g$ is not
finite-state, contradiction.

Suppose $|\lambda|=1$ and $\lambda$ is not a root of unity. As above there is a sequence of letters
$x_n\in X$ such that for each $n$ either $|\xi(g|_{x_1x_2\ldots x_{n+1}})|>|\xi(g|_{x_1x_2\ldots
x_n})|$ or $\xi(g|_{x_1x_2\ldots x_{n+1}})=\lambda\xi(g|_{x_1x_2\ldots x_n})$. In either case we have a
contradiction with the fact that the action is finite-state.
\end{proof}

It is left to prove that Jordan cells for roots of unity have size $1$. Let $m$ be an integer number
such that $\varepsilon^m=1$ for every root of unity $\varepsilon$ from the spectrum of $\phiL$. Then
the spectrum of $\phiL^m$ consists of $1$ and numbers less than $1$ in modulus. The self-similar action
$(G,X^{*})$ over the alphabet $X$ induces the self-similar action $(G,(X^m)^{*})$ over the alphabet
$X^m$ of words of length $m$ over $X$. Moreover, since the action $(G,X^{*})$ is finite-state then
obviously the action $(G,(X^m)^{*})$ is also finite-state. Note that the virtual endomorphism $\phi_v$
of the action $(G,(X^m)^{*})$ associated to a word $v=x_1\ldots x_m\in X^m$ is the composition
$\phi_{v}=\phi_{x_1}\circ\dots\circ\phi_{x_n}$ of virtual endomorphisms of the action over $X$. In
particular, $\phi_{x_1\ldots x_1}=\phi^m$ and the action $(G,(X^m)^{*})$ corresponds to the pair
$(G,\phi^m)$. If we know that the size of Jordan cells of $\phiL^m$ with eigenvalue $1$ have size $1$,
then the same holds for $\phiL$ for roots of unity. Hence we can assume that all roots of unity in the
spectrum of $\phiL$ are equal to $1$.

Recall that $\mathscr{L}_{\mathbb{Q}}=\mbQ\text{-span}(\log G)$ is a rational Lie subalgebra of $\srL$
by Theorem~\ref{bookthm1} such that $\mathscr{L}_{\mathbb{Q}}\otimes{}\mathbb{R}=\mathscr{L}$ (i.e.,
determines a rational structure on $\srL$). It is easy to see that
$\phiL(\mathscr{L}_\mathbb{Q})\subset\mathscr{L}_\mathbb{Q}$. Indeed, since $H$ is of finite index in
$G$ they are commensurable and hence by Theorem \ref{bookthm2} define the same rational structure on
$\srL$. Thus by Theorem \ref{bookthm1} $\mathscr{L}_{\mathbb{Q}}$ is also equal to the
$\mathbb{Q}$-span of vectors from $\log{(H)}$. Since $\phi(H)\subset{}G$ we have that
$\phiL(\mathscr{L}_{\mathbb{Q}})\subset{}\mathscr{L}_{\mathbb{Q}}$ by Equation~(\ref{commi}).

Suppose there is a Jordan cell of $\phiL$ with eigenvalue $1$ that has size greater than $1$. Then
there exist nonzero vectors $v,u\in\mathscr{L}$ such that $\phiL(v)=v$ and $\phiL(u)=u+v$. Since
$\phiL(\mathscr{L}_{\mathbb{Q}})\subset{}\mathscr{L}_{\mathbb{Q}}$, the vectors $v$ and $u$ can be
chosen in $\srL_\mbQ$.  By Theorem~\ref{bookthm3} the group $G$ contains a subgroup $G_0$ of finite
index such that $\log G_0$ is a lattice in $\mathscr{L}_{\mathbb{Q}}$, i.e., $\log G_0$ is closed under
addition and its span over $\mathbb{Q}$ is equal to $\mathscr{L}_{\mathbb{Q}}$. Multiplying $v$ and $u$
by a suitable integer we can assume that they belong to $\log G_0$, and thus $u+nv\in\log G_0$ for all
$n\in\mathbb{N}$. Consider the element $g=\exp(u)\in G$. We get
$$\phi^n(g)=\phi^n(\exp(u))=\exp(\phiL^n(u))=\exp(u+nv)\in G_0\subset G,$$ and hence $\phi^{n-1}(g)\in \phi^{-1}(G)=H=St_G(x_1)$
for all $n\geq 1$. Then the element $g$ fixes the word $x_1x_1\ldots x_1$ ($n$ times) and has the state
$g|_{x_1x_1\ldots x_1}=\phi^n(g)=\exp(u+nv)$ for all $n\geq 1$. Since all elements $u+nv$ are
different, the element $g$ is not finite-state. We got a contradiction.
\end{proof}

\begin{proof}[Proof of Theorem~\ref{thm_main2}]
If the spectral radius of $\phiL$ is less than $1$, then the action is contracting by
Proposition~\ref{prop_contracting} and thus finite-state.

For the converse, it is sufficient to prove that if the Jordan normal form of $\phiL$ satisfies the
condition in Theorem~\ref{thm_main1} and the spectrum of $\phiL$ contains a root of unity then there
exists a non-finite-state action for $(G,\phi)$. As in the previous proof, we can assume that all roots
of unity from the spectrum are equal to $1$, and we find an element $q\in H$ such that $\phi(q)=q$.
Since the $\phi$-core is trivial, the $\phi|_{Z(H)}$-core is trivial too by
Proposition~\ref{prop_Sidki}. Then by \cite[Theorem~2.12.1]{self_sim_groups} the corresponding linear
map $\widehat{\phi|}_{Z(H)}$ has spectral radius less than $1$. Let us choose a set of coset
representatives $D=\{q_1=e,q_2,\ldots,q_{d'}\}$ for $Z(H)$ in $Z(G)$. If every element $g\in Z(G)$ can
be expressed as a product
\begin{eqnarray}
g&=&q_{i_1} \phi^{-1}( q_{i_2} \phi^{-1}(\ldots \phi^{-1}(q_{i_n})\ldots))=\label{eqn_g_expresses as product}\\
&=&q_{i_1}\phi^{-1}(q_{i_2}) \phi^{-2}(q_{i_3})\ldots \phi^{-n+1}(q_{i_n}) \nonumber
\end{eqnarray}
for $q_{i_j}\in D$, then we take $k>1$ such that the set $D^k=\{q_1^k=e,q_2^k,\ldots q_{d'}^k\}$
consists of coset representatives for $Z(H)$ in $Z(G)$ and replace $D$ by $D^k$. In this case every
product in (\ref{eqn_g_expresses as product}) belongs to a proper subgroup $Z(G)^k$ of $Z(G)$. We
complete $D$ to the set of coset representatives of $H$ in $G$ by elements $q_{d'+1},\ldots, q_d$.
Replace the coset representative $q_1=e$ by the element $q\in H$. Let us prove that the associated
self-similar action of the group $G$ is not finite-state.

Take an element $g\in Z(G)$ that cannot be expressed in the form (\ref{eqn_g_expresses as product}),
and consider the state of $g$ at the word $x_1x_1\ldots x_1$ ($n$ times):
\begin{eqnarray*}
g|_{x_1x_1\ldots x_1}=\phi(q^{-1}_{y_n}\ldots \phi(q^{-1}_{y_2} \phi(q^{-1}_{y_1} g q) q) \ldots
q)=\phi(q^{-1}_{y_n}\ldots \phi(q^{-1}_{y_2} \phi(q^{-1}_{y_1} g))\ldots ) q^n,
\end{eqnarray*}
where $g(x_1x_1\ldots x_1)=y_1y_2\ldots y_n$. All elements $q_{y_i}$ are taken from the set
$\{q,q_2,\ldots ,q_{d'}\}$. The elements $\phi(q^{-1}_{y_j}\ldots \phi(q^{-1}_{y_2} \phi(q^{-1}_{y_1}
g))\ldots)$ with all $q_{y_k}$ from $\{q_2,\ldots,q_{d'}\}$ belong to the center $Z(G)$. We can move
every element $q_{y_i}$ that is equal to $q$ to the right and include in the power $q^n$. Hence we can
write the state as
\begin{equation}\label{eqn_g_q^m}
g|_{x_1x_1\ldots x_1}=\phi(q^{-1}_{y_n}\ldots \phi(q^{-1}_{y_2} \phi(q^{-1}_{y_1} g))\ldots ) q^{m},
\end{equation}
with all $q_{y_i}\in \{e,q_2,\ldots,q_{d'}\}$ (here we replace every $q_{y_i}=q$ by $q_{y_i}=e$), and
$m$ is equal to the number of letters $y_i$ not equal to $x_1$. Since $\widehat{\phi|}_{Z(H)}$ is
contracting, all the products $\phi(q^{-1}_{y_n}) \ldots\phi^{n-1}(q^{-1}_{y_2})
\phi^n(q^{-1}_{y_1})\phi^n(g)$ belong to a compact subset of $L$ (see the proof of
Proposition~\ref{prop_contracting}), but also belong to the lattice $G$. Hence these products assume a
finite number of values. Let us analyze the values of~$m$.

Notice that $g$ cannot stabilize the sequence $x_1x_1\ldots$. Indeed, if $g(x_1)=x_1$ then
$g|_{x_1}=\phi(q^{-1}gq)=\phi(g)\in Z(G)$. Hence, if $g(x_1x_1\ldots)=x_1x_1\ldots$ then $\phi^n(g)\in
Z(G)$ for all $n\geq 1$. It implies that there exists a non-trivial normal $\phi$-invariant subgroup of
$Z(G)$ and we get a contradiction with the faithfulness of the action. Suppose $g$ changes only
finitely many letters in the sequence $x_1x_1\ldots$. Then $g|_{x_1x_1\ldots x_1}$ and
$g|_{x_1x_1\ldots x_1}q^{-m}$ stabilize $x_1x_1\ldots$ for long enough word $x_1x_1\ldots x_1$. At the
same time
\[
g|_{x_1x_1\ldots x_1}q^{-m}=\phi(q^{-1}_{y_n}\ldots \phi(q^{-1}_{y_2} \phi(q^{-1}_{y_1} g))\ldots )\in
Z(G).
\]
and we get that this element should be trivial. However in this case $g$ can be expressed in the form
(\ref{eqn_g_expresses as product}), we get a contradiction with the choice of $g$.

Since $g$ changes infinitely many letters in the sequence $x_1x_1\ldots$, the number $m$ in Equation
(\ref{eqn_g_q^m}) goes to infinity as the length of $x_1x_1\ldots x_1$ goes to infinity.  The elements
$q^m$ are all different, and hence $g$ has infinitely many states.
\end{proof}

\section{Example}\label{Section_Example}

Consider the discrete Heisenberg group
\[
G=\left\{  (x,y,z)=\left(
       \begin{array}{ccc}
         1 & x & z \\
         0 & 1 & y \\
         0 & 0 & 1 \\
       \end{array}
     \right) : x,y,z\in\mathbb{Z}
     \right\},
\]
its subgroup $H=\left\{ (x,2y,2z) : x,y,z\in\mathbb{Z}\right\}$, and the isomorphism $\phi:
H\rightarrow G$ given by $\phi(x,y,z)=(x,y/2,z/2)$. One can directly check that the $\phi$-core is
trivial, and therefore every self-similar action for the pair $(G,\phi)$ is faithful (we can also
notice that $\phi|_{Z(H)}:Z(H)\rightarrow Z(G)$ has spectral radius $\frac{1}{2}$ and thus self-similar
actions for $(Z(G),\phi|_{Z(H)})$ are faithful, and hence the same holds for $(G,\phi)$ by
Proposition~\ref{prop_Sidki}).

The matrix of $\phi$ is diagonal with eigenvalues $1, \frac{1}{2}, \frac{1}{2}$, and Theorems \ref{thm_main1}
and~\ref{thm_main2} imply that there exist both finite-state and non-finite-state self-similar actions for the pair
$(G,\phi)$. First, let us construct a finite-state action. Choose coset representatives $D=\{(0,0,0), (0,1,0), (0,0,1),
(0,1,1)\}$ and consider the associated self-similar action $(G,X^{*})$ over the alphabet $X=\{1,2,3,4\}$ given by
Equation~(\ref{eq_action_given_by_virt_end}). The action of the generators $a=(1,0,0)$ and $b=(0,1,0)$ of the group
satisfies the following recursions:
\begin{align*}
a(1v)&=1a(v) & a(2v)&=4a(v) & a(3v)&=3a(v) & a(4v)&=2(b^{-1}ab)(v)\\
b(1v)&=2v    & b(2v)&=1b(v) & b(3v)&=4v    & b(4v)&=3b(v)
\end{align*}
The elements $a$ and $b$ are finite-state, namely the states of $a$ are $a, b^{-1}ab, b^{-2}ab^2$ and the states of $b$
are $e,b$. Hence the action $(G,X^{*})$ is finite-state.

Let us change the set of coset representatives and choose $D'=\{(1,0,0),\linebreak (0,1,0), (0,0,1), (0,1,1)\}$. Then the
action of the generators of the group $G$ satisfies the recursions
\begin{align*}
a(1v)&=1a(v) & a(2v)&=4a(v) & a(3v)&=3a(v) & a(4v)&=2(b^{-1}ab)(v)\\
b(1v)&=2a(v) & b(2v)&=1(a^{-1}b)(v) & b(3v)&=4v    & b(4v)&=3b(v)
\end{align*}
In order to see that this action is not finite-state, let us look at the action of the element
$c=(0,0,1)=a^{-1}b^{-1}ab$:
\begin{align*}
c(1v)&=3 (a^2b^{-1}a^{-1}b)(v) & c(2v)&=4c(v) & c(3v)&=1a^{-1}(v) & c(4v)&=2c^2(v)
\end{align*}
The element $c^2$ is a state of $c$, and $c^{2n}|_2=c^{3n}$, $c^{2n+1}|_2=c^{3n+1}$, $c^{2n+1}|_4=c^{3n+2}$. Inductively
we get that all powers $c^n$ for $n\in\mathbb{N}$ are the states of $c$ and hence the element $c$ is not finite-state.

\end{document}